\documentclass[leqno,12pt]{article}

\usepackage{epsfig}

\usepackage{amsmath}
\usepackage{amscd}
\usepackage{amsopn}
\usepackage{amsthm}
\usepackage{amsfonts,amssymb}\usepackage{amsfonts,bbm}
\usepackage{latexsym}

\usepackage{color}

\makeatletter
\@addtoreset{equation}{section}
\makeatother

\newtheorem{lemma}{LEMMA}[section]
\newtheorem{proposition}[lemma]{PROPOSITION}
\newtheorem{corollary}[lemma]{COROLLARY}
\newtheorem{theorem}[lemma]{THEOREM}
\newtheorem{remark}[lemma]{REMARK}

\newtheorem{examples}[lemma]{EXAMPLES}

\newcommand{\real}{\mathbbm{R}}

\newcommand{\nat}{\mathbbm{N}}

\renewcommand{\a}{\alpha}
\renewcommand{\b}{\beta}

\newcommand{\g}{\gamma}

\newcommand{\ve}{\varepsilon}

\newcommand{\reald}{{\real^d}}

\newcommand{\und}{\quad\mbox{ and }\quad}
\newcommand{\inv}{^{-1}}
\newcommand{\ov}{\overline}

\newcommand{\W}{\mathcal W}  
\newcommand{\C}{\mathcal C}  

\newcommand{\F}{\mathcal F}

\renewcommand{\H}{{\mathcal H}}
\newcommand{\B}{\mathcal B}

\newcommand{\M}{\mathcal M}

\newcommand{\osc}{\operatorname*{osc}}

\newcommand{\itemframe}%
{\setlength{\parskip}{10pt}\begin{enumerate} \setlength{\topsep}{10pt}%
\setlength{\itemsep}{15pt}\setlength{\parsep}{5pt}}

\newcommand{\mx}{\mu_x}
\newcommand{\my}{\mu_y}

\newcommand{\vx}{\ve_x}

\newcommand{\Uo}{\mathcal U_0}

 \title{Intrinsic  H\"older continuity of harmonic functions}
\author{Wolfhard Hansen} 
\date{}

\begin{document}
\maketitle

\begin{abstract}
In a setting, where only ``exit measures'' are given,  as they are  associated with a  right continuous 
strong Markov process on  a separable metric space, we provide simple criteria for scaling invariant H\"older 
continuity of bounded harmonic functions 
with respect to a distance function which, in applications, may  be adapted to the special situation.
In particular, already a~very weak scaling property ensures that Harnack inequalities imply H\"older continuity.
Our approach covers recent results by M.\ Kassmann and A.\ Mimica as well as cases, where a Green function
leads to an intrinsic metric.

 
  \end{abstract}

 \section{Harmonic functions in a general  setting}

During the last years,  H\"older continuity of bounded harmonic functions 
has  been studied for various classes of Markov processes (see \cite{kassmann-mimica-final,H-hoelder} 
and the references therein). The aim of this paper  
is   to offer not only a~unified (and perhaps more transparent) approach   to  results obtained until now, 
but also the possibility for   applications  in new cases.
\footnote{The final publication is available at Springer via http://dx.doi.org/10.1007/s1118-016-9604-8.}

Let $X$ be a topological space such that   finite measures~$\mu$ on
its $\sigma$-algebra~$\B(X)$ of Borel subsets satisfy 
\[
\mu(A)=\sup\{\mu(F)\colon F \mbox{ closed, } F\subset A\}, \qquad A\in \B(X).
\]
This holds if $X$ is a separable metric space
(on its completion   every finite measure is tight). 
Let $\M(X)$ denote the set of all finite measures on $(X,\B(X))$
(which we also consider as measures on the $\sigma$-algebra $\B^\ast(X)$ of all
universally measurable sets).  
Given a~set $\F$ of numerical functions on~$X$, let~$\F_b$, $\F^+$ be the set of all functions in~$\F$
which are bounded, positive respectively.

For great flexibility in applications, we consider an open neighborhood $X_0$ of a~point $x_0\in X$
and suppose that we have  a continuous real function $\rho_0\ge 0$ on $X_0$
with $\rho_0(x_0)=0$ and $0<R_0\le \infty$ such that, for every $0<r<R_0$, the closure
of 
\begin{equation*}
    U_r:=\{x\in X\colon \rho_0(x)<r\}                   
\end{equation*} 
 is contained in $X_0$.
Let $\Uo$ denote the set of all open sets $V$ in $X$ with $V\subset U_r$ for some~$0<r<R_0$. 

We suppose that we have    measures~$\mu_x^U\in\M(X)$, $x\in X$,  $U\in \Uo$,  such that the following hold
for all $x\in X$ and $U,V \in \Uo$ (where $\vx$ is  the  Dirac measure at~$x$):
{\it
\begin{itemize}
\item[\rm(i)]
 The measure $\mu_x^U$ is supported by $U^c$, and $\mu_x^U(X)=1$.  If $x\in U^c$, then $\mu_x^U=\ve_x$.
\item[\rm(ii)]
The  functions $y\mapsto \mu_y^U(E)$, $E\in\B(X)$, are universally measurable on $X$ and
\begin{equation}\label{it-bal} 
\mu_x^U=(\mu_x^{V})^U := \int \mu_y^U\,d\mu_x^{V}(y), \qquad \mbox{ if } V\subset U.
\end{equation} 
\end{itemize} 
}

Let us note that, having (i),  the equality in  (\ref{it-bal}) amounts to 
\begin{equation}\label{it-bal-det}
    \mu_x^U=\mu_x^V|_{U^c} + \int_U \mu_y^U\,d\mu_x^V(y). 
\end{equation} 

Of course, stochastic processes and  potential theory  abundantly provide examples (with $X_0=X$,
 $\rho_0(x)=\rho(x,x_0)$, where $\rho$ is \emph{any}  metric for the topology of~$X$).

\begin{examples}\label{hunt-bal}{\rm
1. Right process $\mathfrak X$ with strong Markov property  on a Radon space $X$,
\[
    \tau_U:=\inf\{t\ge 0\colon X_t\in U^c\}< \infty \mbox{\ $\mathbbm P^x$-a.s.} \und \mu_x^U(E):=P^x[X_{\tau_U}\in E]
\]
for all $U\in \Uo$, $x\in X$, $E\in \B(X)$  (\cite[Propositions 1.6.5 and 1.7.11, Theorem~1.8.5]{beznea-boboc}).

If $U,V\in \Uo$ with $V\subset U$, then $\tau_U=\tau_V+\tau_U\circ \theta_{\tau_V}  $,
and hence, by the strong Markov
property, for all $x\in X$ and $E\in \B(X)$, 
\begin{equation*}  
\mx^U(E)=\mathbbm P^x[X_{\tau_U}\in E]=\mathbbm E^x\left(\mathbbm
P^{X_{\tau_V}} [X_{\tau_U}\in E]\right)=\int \my^U(E)\,d\mx^V(y). 
\end{equation*}

2. Balayage space $(X,\W)$ (see {\rm \cite{BH}})  such that $1\in \W$,  
\[
               \int v\,d\mu_x^U= R_v^{U^c}(x):=\inf\{ w(x)\colon w\in \W,\ w\ge v \mbox{ on } U^c\}, \qquad v\in \W.
\] 
The properties {\rm(i)} and {\rm(ii)} follow from \cite[VI.2.1, 2.4, 2.10, 9.1]{BH}. 
}
\end{examples}

Given $U\in \Uo$, let  $\H(U)$ denote the set of all universally measurable real  functions~$h$ on~$X$ 
which are \emph{harmonic on $U$}, that is, for all   open
sets $V$ with $\ov V\subset U$ and all~$x\in V$, are  $\mu_x^V$-integrable and satisfy
\begin{equation}\label{har-def}
                     \int h\,d\mu_x^V=h(x).
\end{equation}
Obviously, constant functions are harmonic on $U$ and,  for every bounded Borel measurable function   $f$
on $X$, the function $x\mapsto \int f\,d\mu_x^U$ is harmonic on U, by (\ref{it-bal}).
The latter even holds for every bounded \emph{universally} measurable function~$f$ on~$X$ (see
\cite[Section 2]{HN-scaling-harnack}).

\section{Main result}

Our aim is to obtain criteria for the following scaling invariant H\"older
continuity of bounded harmonic functions. 

\begin{itemize} 
\item[\rm (HC)] 
There exist $C>0$ and $\b\in (0,1)$
 such that, for all $0<r<R_0$, 
\begin{equation*} 
                     |h(x)-h(x_0)|\le C \|h\|_\infty \biggl(\frac{\rho_0(x)} {  r}\biggr)^\b  \quad\mbox{ for all }
                                            h\in \H_b(U_r) \mbox{ and } x\in U_r.
 \end{equation*} 
\end{itemize} 
To that end we introduce the following   properties. 

\begin{itemize}
\item[\rm (J$_1$)]
There are $\a,\delta_0\in (0,1)$ such that, 
for every  $0<r<R_0$ and every  universally~measurable set $A$ in $U_r$, 
\begin{equation}\label{alternative}
\inf\{  \mu_x^{U_{\a r}}(  A) \colon x\in U_{\a^2 r}\}> \delta_0 
             \quad\mbox { or } \quad
\inf\{  \mu_x^{U_{\a r}}(U_r\setminus A) \colon x\in U_{\a^2 r}\}> \delta_0.
\end{equation} 
\item[\rm (J$_2$)] 
There are $\a_0,a_0\in (0,1)$ and $C_0\ge 1$ such that, for all $0<r<R_0$ and  $n\in\nat$,
\begin{equation}\label{jump-m}
              \mu_x^{U_{\a_0^n r}}(U_r^c)\le C_0a_0^n\qquad \mbox{ for every }x\in U_{\a_0^{n+1}r}.
\end{equation} 
\end{itemize}
 
Of course,  (J$_2$) holds trivially if the harmonic measures
$\mu_x^U$, $ U\in \Uo$,  $x\in U$,
are supported by the boundary of $U$, that is, in the Examples \ref{hunt-bal}, if $\mathfrak X$ is a~diffusion
or~$(X,\W)$ is a harmonic space (since then  $\mu_x^{U_{\a^n r}}(U_r^c)=0$ for  $x\in U_{\a^n r}$).

\begin{lemma}\label{J-strong}
If {\rm (J$_2$)} holds, then, for \emph{every} $a\in (0,1)$, there
exists $ \a\in (0,1)$ 
such that, for all $0<r<R_0$ and $n\in\nat$,
\begin{equation}\label{essential-jump}
              \mu_x^{U_{  \a^n r}}(U_r^c)\le a^n \qquad \mbox{ for
                every } x\in U_{ \a^{n+1} r}.
\end{equation} 
\end{lemma} 

\begin{proof} If (J$_2$) holds with $\a_0, a_0, C_0$, 
 we  fix $k\in\nat$  with $a_0^k<C_0\inv a$, and  take $ \a:=\a_0^k$. 
Then, for all   $n\in\nat$ and $x\in U_{\a^{n+1} r} \subset U_{\a_0^{kn+1}r}$,
\begin{equation*}
\mu_x^{U_{\a^n r}}(U_r^c)\le C_0 a_0^{kn} \le C_0 (C_0\inv
a)^n\le a^n.
\end{equation*} 
\end{proof} 

\begin{remark}\label{smaller-a}
{\rm 
If (J$_1$), (J$_2$) or (\ref{essential-jump}) hold for some $\a\in (0,1)$, then
they hold for any $\tilde \a\in (0,\a)$ (keeping the other constants).

Indeed, for (J$_2$) and (\ref{essential-jump}), this is true, since 
$\mu_x^{U_{ \a^n r}}(U_r^c)\ge\mu_x^{U_{\tilde \a^n r}}(U_r^c)$,  by (\ref{it-bal-det}).
For ($J_1$),  it suffices to observe that, defining  $\tilde r:=(\tilde \a/\a)r$, trivially 
 \begin{equation*} 
   \mu_x^{U_{\tilde a r}}(A)=   \mu_x^{U_{a \tilde r}}(A)\ge \mu_x^{U_{a \tilde r}}(A\cap  U_{\tilde r}).
\end{equation*} 
}
\end{remark}

\begin{theorem}\label{hoelder} 
Suppose   {\rm(J$_1$)} and {\rm (J$_2$)}. 
Then {\rm (HC)} holds
{\rm (}with $C$ and $\b$   which depend only on the  constants
in  {\rm(J$_1$)} and {\rm (J$_2$)}{\rm)}.
\end{theorem}

  \begin{proof}[Proof {\rm (cf.\ the proofs of
      \cite[Theorem 4.1]{bass-levin} and \cite[Theorem~1.4]{kassmann-mimica-final})}]
Let $\delta_0\in (0,1)$ satisfy~(J$_1$). We define   $ \delta:=\delta_0/6$, $a:=\delta/2$. Then 
$a<(1-a)\delta$, and there exists $1<b< \sqrt{3/2} $ with
\begin{equation}\label{def-ab}
      b^3(1-3\delta)\le 1-2\delta   \und a b^4<(1-ab)\delta.
\end{equation}
By Lemma \ref{J-strong} and Remark \ref{smaller-a}, there exists $\a\in (0,1)$
such that the statement in~(J$_1$) holds with this $\a$ and $\delta_0$, and
\begin{equation}\label{jump-2}
   \mu_x^{U_{\a^k r}}(U_r^c)\le a^k,\qquad  \mbox{$0<r<R_0$,  $k\in \nat$, $x\in U_{\a^{k+1} r}$.}
\end{equation} 
 We now fix  $0<r<R_0$  and $h\in\H_b(U_r)$ with   $\|h\|_\infty=1$. 
For $n=0,1,2,\dots$  let  
\[
      B_n:= U_{\a^n r} , \qquad M_n(h):=\frac12 (\sup h( B_n)+\inf h(B_n)),
\]
\[
\osc\nolimits_n(h):=\sup h(B_n)-\inf h(B_n)=\sup\{h(x)-h(y)\colon x,y\in B_n\},
 \]
Clearly,  $M_n(-h)=-M_n(h)$ and $\osc_n(-h)=\osc_n(h)$. We claim that, for  $n\ge 0$, 
\begin{equation}\label{hoelder-ess}
        \osc\nolimits_n(h)\le s_n:=3 b^{-n}. 
\end{equation} 
Clearly, (\ref{hoelder-ess})  holds trivially for $n=0,1,2$,  since $\osc_n(h)\le 2$ and $b^2<3/2$.

Let  us  consider $n\ge 2$ and suppose that  (\ref{hoelder-ess}) holds for all $0\le j\le n$. Let  $S_n:=B_n\setminus B_{n+1}$
and
\[
     A(h):=  \{z\in S_n\colon h(z)\le M_n(h)\}, \qquad
     B(h):=\{z\in S_n\colon h(z)\ge M_n(h)\}. 
\] 
By (J$_1$),  
\begin{equation*} 
\inf\{ \mu_x^{B_{n+1}}(  A(h))\colon x\in B_{n+2}\}  > \delta_0
\quad\mbox{ or } \quad         
\inf\{ \mu_x^{B_{n+1}}(  B(h))\colon x\in B_{n+2} \}  > \delta_0.
\end{equation*} 
Without loss of generality $\inf\{ \mu_x^{B_{n+1}}(  A(h))\colon x\in B_{n+2} \}  > \delta_0$
(otherwise we   replace~$h$ by $-h$ using $B(h) =A(-h)$). 

Now let us fix two points $x,y\in B_{n+2}$. We  claim that 
\begin{equation}\label{to-show}
h(x)-h(y)\le s_{n+1}.
\end{equation} 

We may choose a closed set $F$
in $ A(h)$ such that $\mu_x^{B_{n+1}}(F)>\delta_0$. Then  
\begin{equation*} 
\mu:=\mu_x^{B_n\setminus F} 
\end{equation*} 
satisfies $\mu(F)>\delta_0=6\delta$, by (\ref{it-bal-det}).
Since $h\in \H_b(U_r)$ and $\mu$ is a probability measure, 
\begin{equation}\label{mu-equation}
           h(x)-h(y)=\int  (h-h(y))\,d\mu.
\end{equation} 
The measure $\mu$  is supported by $F\cup B_n^c$.
Clearly,
\begin{equation*} 
    \int_F (h-h(y))\,d\mu\le \bigl(M_n(h) -\inf h(B_n)\bigr)
    \mu(F)=\frac 12 \osc\nolimits_n(h) \mu(F)<
\frac 12 s_{n-1} \mu(F).
\end{equation*} 
Since $\mu(B_{n}^c)=1-\mu(F)$, we see that
\begin{equation*} 
           \int_{B_{n-1}\setminus B_{n}} (h-h(y)) \,d\mu \le s_{n-1} (1-\mu(F)).
\end{equation*} 
   Combining the two previous estimates  we obtain that 
\begin{equation}\label{mu-ess-2}
\int_{F\cup (B_{n-1}\setminus B_{n})} (h-h(y))\,d\mu \le s_{n-1} (1-\frac 12 \mu(F)) \le s_{n+2}   (1-2\delta) ,
\end{equation} 
where the last inequality follows from $\mu(F)>6\delta$ and (\ref{def-ab}).
Moreover 
\begin{equation}\label{mu-ess-3}
          \int_{B_{n-1}^c} (h-h(y))\,d\mu\le  2\mu(B_0^c) +\sum\nolimits_{j=0}^{n-2} s_j  \mu(B_j \setminus B_{j+1}). 
\end{equation} 
By (\ref{it-bal-det}) and  (\ref{jump-2}) (applied to $\a^mr$ in place of $r$),  for all $0\le m\le n$,
\[ 
         \mu(B_m^c)\le \mu_x^{B_n}(B_m^c)\le a^{n-m}.
\]
By (\ref{def-ab}), $ab^4<(1-ab)\delta$. Hence  
 $2\mu(B_0^c)\le 2a^n\le 2(b^{-4}\delta)^n< s_{n+2} \delta$ and
\[
      \sum\nolimits _{j=0}^{n-2} s_j \mu(B_j\setminus B_{j+1}) \le \sum\nolimits _{j=0}^{n-2} s_j \mu( B_{j+1}^c) 
\le     3  \sum\nolimits _{j=0}^{n-2} b^{-j} a^{n-(j+1)}
                =  s_{n+2} s, 
\]
where
\[     s =b^{n+2} \sum\nolimits_{j=0}^{n-2} b^{-j} a^{n-(j+1)}=b^3 \sum\nolimits_{j=0}^{n-2}   (ab)^{n-(j+1)}
<  \frac{a b^4} { 1-ab} < \delta.
\]
Having (\ref{mu-equation}), the estimates  (\ref{mu-ess-2}) and
(\ref{mu-ess-3}) hence
  yield   $h(x)-h(y)\le s_{n+2}$. So  (\ref{to-show})~holds, and we
  see that  $\osc_{n+2}(h)\le s_{n+2}$.

 Given $x\in B_0\setminus \{x_0\}$, there exists $n\ge 0$ 
such that $x\in B_n\setminus B_{n+1}$. Defining  $\b:= (\ln b)/\ln (1/\a)$, we finally conclude that
\[
               |h(x)-h(x_0)|\le  3b^{-n} = 3\a^{ n\b} \le 3 \biggl(\frac{\rho_0(x)}{\a r}\biggr)^\b.
\]
\end{proof} 

\begin{remark}\label{nec} {\rm
The proof shows that in  dealing with harmonic functions which are Borel
measurable, continuous, respectively, we need (J$_1$) only for sets
$A$ in $U_r$ which are Borel measurable, relatively closed in $U_r$,
respectively.
}
\end{remark}

To see that (J$_1$) is almost necessary for H\"older continuity
of bounded harmonic functions we introduce the following 
weak property which  immediately follows from both             
(J$_1$) and (J$_2$) and merely states that $\rho_0$ provides a
suitable scaling  at $x_0$. 

\begin{itemize}
\item[\rm (J$_0$)] There are $\a, \delta_0\in (0,1)$ such that, for every $0<r<R_0$,
\begin{equation}\label{jump-0} 
         \mu_{x_0} ^{U_{\a r}} (U_r)>\delta_0.
\end{equation} 
\end{itemize} 
For the moment, let us fix $0<r<R_0$, $\a\in (0,1)$, 
let $S:=U_r\setminus U_{\a r}$ and    $A$~be a~universally
    measurable set in $X$. 
Of course, $\mu_{x_0}^{U_{\a r}}(U_r)=\mu_{x_0}^{U_{\a r}}(S)$, 
hence  (\ref{jump-0}) implies that
    $\mu_{x_0}^{U_{\a r}} (S\cap A)> \delta_0/2$ or
    $\mu_{x_0}^{U_{\a r}} (S\setminus  A)> \delta_0/2$, and  
    there exists a~closed set~$F$ in~$S\cap A$ or~$S\setminus A$ such
    that 
\begin{equation}\label{j0F}
\mu_{x_0}^{U_{\a r}}(F)>\delta_0/2.
\end{equation} 

\begin{proposition}\label{necessary}
Assuming {\rm (J$_0$)}, property {\rm (J$_1$)} is necessary for  {\rm (HC)}.
\end{proposition} 

\begin{proof} Suppose that (HC) and (J$_0$) hold with $C,\b,\a_0, \delta_0$.  We choose $0<\a<\a_0$ such that $C \a^\b<\delta_0/4$.
Let $0<r<R_0$ and $F$ be a closed set in $U_r\setminus U_{\a r}$ such that (\ref{j0F}) holds.
The  function $x\mapsto \mu_x^{U_{\a r}}(F)$ is harmonic 
 on $U_{\a r}$. So (HC) implies that, for every $x\in U_{\a^2 r}$,
 \begin{equation*} 
        \mu_x^{U_{\a r}}(F)\ge  \mu_{x_0}^{U_{\a  r}}(F)-C \a^\b >\delta_0/4. 
  \end{equation*} 
\end{proof}

\begin{corollary}\label{diffusion-hoelder}
Suppose that the   measures $\mu_x^U$, $x\in U\in \Uo$, are supported
by the boundary of $U$. Then {\rm(HC)} holds if and only if
{\rm(J$_1$)} holds.
\end{corollary} 

Moreover, Theorem \ref{hoelder} will quickly lead to Corollary \ref{harnack-hoelder} which, in turn,
yields H\"older continuity of bounded harmonic functions and continuity of harmonic functions
provided there is a suitable associated Green function (see Remark \ref{harnack-appl}). 
For a direct application of Theorem \ref{hoelder} 
in the setting of \cite{kassmann-mimica-final} see  Sections  \ref{application} and \ref{examples}.

\section{Harnack inequalities imply H\"older continuity}

Next let us see that the following scaling invariant Harnack
inequalities are sufficient for (HC)  provided (J$_0$) holds. 

\begin{itemize} 
\item[\rm (HI)]  Harnack inequality:
There exist $\a\in (0,1)$ and $K\ge 1$ such that
\begin{equation*} 
              \sup h(  U_{\a r})\le K\inf h( U_{\a r})\qquad\mbox{ for all  } 0<r<R_0 \mbox{ and } h\in \H_b^+(U_r).
\end{equation*} 
\end{itemize} 

\begin{proposition}\label{suff-j}
 If {\rm(J$_0$)} and {\rm(HI)} hold, then {\rm (J$_1$)}  and {\rm (J$_2$)} are  satisfied. 
\end{proposition}

\begin{proof} Let $\a,\delta_0\in (0,1)$ satisfy  (J$_0$) and (HI), and let  $0<r<R_0$.

1) Let $F$ be a closed set in  
$ U_r\setminus U_{\a r}$ such that (\ref{j0F}) holds. By (HI), the harmonicity
of the function $x\mapsto \mu_x^{U_\a r}(F)$ on $U_{\a r}$ yields
 that $\mu_x^{U_{\a r}} (F)>(2K)\inv \delta_0$ for every $x\in
U_{\a^2 r}$. So (J$_1$) holds (with $(2K)\inv \delta_0$ in place of $\delta_0$).

2) Let $0<s\le r$. By (J$_0$),  
there exists a closed set $F$ in $U_{\a s}\setminus U_{\a^2 s}$ such
that  $\mu_{x_0}^{U_{\a^2 s}}(F)> \delta_0$.
 Let us fix $x\in U_{\a^3 r}$. By (HI) and  (\ref{it-bal-det}),
\begin{equation*}
K \mu_x^{U_s\setminus F}(F)\ge   \mu_{x_0}^{U_s\setminus F}(F)\ge 
  \mu_{x_0}^{U_{\a^2 s}}(F) >   \delta_0. 
\end{equation*} 
By (\ref{it-bal-det}),
\[
\mu_x^{U_{\a^2 s}}(U_r^c)\le \mu_x^{U_s\setminus F}(U_r^c) =    
  \mu_x^{U_{s}} (U_r^c)- \int_F \mu_y^{U_{s}}(U_r^c)\,d\mu_x^{U_s \setminus F}, 
\]
where, for every $y\in F$,  $\mu_y^{U_s} (U_r^c)\ge K\inv  \mu_x^{U_s}
(U_r^c) $, by (HI). 
Therefore
\[
\mu_x^{U_{\a^2 s} }(U_r^c)\le \bigl(1-K\inv \mu_x^{U_s\setminus F}(F)\bigr) \mu_x^{U_s}  (U_r^c)
\le \bigl(1-K^{-2}\delta_0\bigr) \mu_x^{U_s}(U_r^c). 
\]
Proceeding by induction,  we get (J$_2$) 
with $a:=1-K^{-2}\delta_0$, $C_0=1$ (since $\mu_x^{U_r}(U_r^c)=1$) and $\a^2$ in place of $\a$.
\end{proof} 

Thus Theorem \ref{hoelder} leads to the following result (where we might recall that (J$_0$)
trivially holds if, for every $0<r<R_0$,  the measure $\mu_{x_0}^{U_r}$ is supported by~$\partial U_r$). 

\begin{corollary}\label{harnack-hoelder}
  {\rm(J$_0$)} and {\rm(HI)} imply {\rm (HC)}. 
\end{corollary} 

\begin{remark}{\rm
\label{harnack-appl} For applications, where properties of an associated Green function imply (HI),
see \cite[Theorems 4.12, 5.2, 6.2, 6.3 and 7.3]{HN-scaling-harnack}.
}
\end{remark}

\section{A general application using  the Dynkin formula\\ 
 and the L\'evy system formula}\label{application}

In this section we shall present  a consequence of Theorem \ref{hoelder} which can  immediately be applied to 
the setting considered in    \cite{kassmann-mimica-final} (see Section \ref{examples}).

Let $X=\reald$, $d\ge 1$, and, for  $x_0\in\reald$ and $0<r\le \infty$, let
\begin{equation*} 
   B(x_0, r):=\{x\in\reald\colon |x-x_0|<r\}, \qquad B_r:=B(0,r).
\end{equation*} 
Let us  fix  $ K_0,  c_0,c_1,c_3 \in (1,\infty)$. 
Further,  let  $0<R<R_0\le \infty$ and  $U_0:=B_{2R_0}$. 
We assume that we have a  measurable function 
$K\colon U_0\times \reald\to [0,\infty)$ and a~continuous function $l\colon (0,R_0)\to (0,\infty)$ 
such that the following hold.
\begin{itemize}
\item[\rm (K)] 
For all $x\in B_{2R}$ and $h\in B_1$, $K(x,h)=K(x,-h)$, and 
\begin{equation*} 
\sup\nolimits_{x\in U_0} \int_{\reald} (1\wedge |h|^2)K(x,h)\, d h \le K_0.
\end{equation*} 
\item[\rm(L$_0$)] For all $x\in U_0$ and $h\in\reald$ with $|h|<R_0$,
\begin{equation*}
   c_0\inv   |h|^{-d} l(|h|)\le    K(x,h) \le c_0|h|^{-d} l(|h|) . 
\end{equation*} 
\item[\rm(L$_1$)] 
         For all $0<r\le s<R$,   
\begin{equation*} 
 l(r/2) \le c_1\,l(r) \und s^{-d} l(s)\le c_1 r^{-d} l(r).
\end{equation*} 
\item[\rm (L$_2$)] 
Defining $  L(r):= \int_r^{R_0} u\inv l(u)\,du$, $0<r\le R_0$, we have $L(0)=\infty$ and
\begin{equation*} 
\tilde L(r):= r^{-2}\int_0^r u l(u)\,du\le c_2   L(r) \qquad\mbox{ for every } 0<r<  R.
\end{equation*}  
\item[\rm(L$_3$)]   L(R/2) +$ (1\vee R^{-2})K_0\le c_3L(R)$.
\end{itemize} 
Moreover, we suppose that there exists a strong Markov process $\mathfrak X=(X_t,\mathbbm P^x)$ 
on~$\reald$ with trajectories
that are right continuous and have left limits and such that,
for all~$x_0\in B_R$, $0<r<R$ and $x\in B(x_0,r)$, the following holds for every $t>0$ and 
\[
          \tau_r:=\inf\{ u\ge 0\colon X_u\notin B(x_0,r)\}.
\]
\begin{itemize} 
\item[\rm(D)] Dynkin   formula: 
For all $f\in \C^{\infty}(\reald)$ with  compact support,
\begin{equation*} 
    \mathbbm E^{x} f(X_{\tau_r\wedge t} )-f(x)=\mathbbm E^{x}\int_0^{\tau_r\wedge t} \int_{\reald}
 \bigl(f(X_u+h)-f(X_u)\bigr)K(X_u,h)\,dh\,du.\footnote{A term $\langle \nabla f(X_u),  h \rangle \mathbbm 1_{\{|h|<1\}} K(X_u,h)$
the reader may expect in the integral on $\reald $  does not yield any contribution because of $K(\cdot,h)=K(\cdot,-h)$.}
\end{equation*} 
\item[\rm(LS)] L\'evy system formula:
For all  Borel sets  $A$ in $B(x_0,r)^c$, 
\begin{equation*} 
\mathbbm P^x[X_{\tau_r\wedge t}\in A] =\mathbbm E^{x} \int_0^{\tau_r\wedge t}  \int_A K(X_u, z-X_u)\,dz\,du.
\end{equation*} 
\end{itemize} 

The existence of such a process is assured if $K(x,h)$ does not depend
on $x$; in the general case it has been established in various contexts
(see the discussion in~\cite{AK09}). 

The only reason for assuming the weird condition (L$_3$)  is that we then may 
stress that  constants $\b$, $C$ and $C_j\in (1,\infty)$, $1\le j\le 5$,  introduced later~on, are valid for  all~$R_0$, $R$, 
$K$ and $l$ satisfying   (K), (L$_0$) -- (L$_3$), (D) and (LS).

Let us observe right away that $\int_{\{r<|h|<R_0\}} K(x_0,h)\,dh
\approx L(r)$,  by (L$_0$), and hence~(K) implies that  $0<L(r)<\infty$ on $(0,R_0)$,
and $L$ is strictly decreasing and continuous on $[0,R_0]$ (with $L(R_0)=0$).

 We claim that Theorem \ref{hoelder}   leads to a result, which  immediately implies the  statement of  Theorem 3 
 in the case $f=0$ and Theorem 12 in \cite{kassmann-mimica-final} (see Section \ref{examples} and 
Corollary \ref{km-corollary}).

\begin{theorem}\label{hoelder-levy}
There exist $C>0$ and   $\b\in (0,1)$ such that,  
for all~$x_0\in B_R$, $0<r<R$ , 
$h\in \H_b(B(x_0,r))$ and $x\in B(x_0,r)$, 
\begin{equation*} 
                     |h(x)-h(x_0)|\le C\|h\|_\infty \biggl(\frac{L(|x-x_0|)} {L(r) }\biggr)^{-\b} .
\end{equation*} 
\end{theorem}

The proof of our claim will be based on the next proposition 
which essentially consists of  rearranged results from  \cite[Section 6]{kassmann-mimica-final}.

Let $x_0\in B_R$ and, for every $r>0$, 
\[
                       V_r:= B(x_0,r)  \und \tau_r:=\inf\{u\ge 0\colon X_u\notin V_r\}.
\]
Moreover, let $\kappa_d$ denote the surface measure of~$B_1$, 
let $k(u):=u^{-d} l(u)$
and let $\mu$ be the measure on~$V_{R_0}$
having  density $  k(|x-x_0|)/L(|x-x_0|)$ with respect to Lebesgue measure.

\begin{proposition}\label{6-km}
There 
are $C_1, C_2, C_3\in (1,\infty)$ such that the following holds.
\begin{itemize} 
\item[\rm (1)]
Let $0<r<R$ and $x\in V_r$. Then
\begin{equation}\label{tau-upper}
E^x\tau_r\le C_1 L(r)\inv. 
\end{equation} 
If $r<s<R$ and $a:=L(r)/L(s)$,
then, for every   Borel set~$A$ in~$S:=V_s\setminus V_r$,
 \begin{equation}\label{j2}
     \mathbbm  P^x[X_{\tau_r}\in A  ]\ge  
        C_1\inv  \frac {\ln a}a\,\frac{ \mu(A)} {\mu(S)}\, L(r) E^x\tau_r.
\end{equation} 
\item[\rm (2)]
If $r,s\in (0,R)$ such that $r\le s/2$, then, for every $x\in V_r$, 
\begin{equation*} 
  \mathbbm P^x[X_{\tau_r}\notin V_s]\le C_2\, L(s) E^x\tau_r.
\end{equation*} 
\item[\rm (3)]
If   $0<r<R$ and $x\in V_{r/2}$,  then 
$ E^{x}\tau_r\ge  C_3\inv L(r)\inv$.
\end{itemize}
\end{proposition} 

Let us note that  (1)   states what can be obtained from  the proof of
\cite[Proposition~17]{kassmann-mimica-final}, 
which is based on (LS).
Hence a separate proof of the first part of \cite[Proposition~15]{kassmann-mimica-final}, giving
an upper estimate of~$\mathbbm E^x\tau_r$ by a multiple of  $ L(r)\inv$, is not needed. 
Moreover,   (2) is \cite[Proposition 16]{kassmann-mimica-final} 
(having a proof using (LS) and the upper estimate for~$\mathbbm E^x\tau_r$).
Finally, a~modification of the proof (the only one using (L$_2$) and
(D)) given in \cite[Proposition 14]{kassmann-mimica-final}  for 
an upper estimate of $t\inv \mathbbm P[\tau_r\le t]$ 
by a~multiple of $L(r)$  directly yields  (3)
so that also the second part of \cite[Proposition~15]{kassmann-mimica-final} is obtained
without an additional proof. 

Because of these   simplifications 
let us write down a complete proof for Proposition \ref{6-km}.
We first establish two simple facts 
which are  repeatedly used also in~\cite{kassmann-mimica-final}.

\begin{lemma}\label{L-doubling} 
Let $C_4:=1+c_1+c_3$. Then $L(r/2)\le C_4 L(r)$ for every $0<r<R$.
\end{lemma}

\begin{proof}  Let us first consider $0<r\le R/2$.
Then $L(r/2)=L(r)+I_r$, where, by~(L$_1$), 
\[
I_r= \int_{r/2}^{r} u\inv l(u)\,du\le 
 c_1\int_{r/2}^{r} u\inv l(2u)\,du= c_1 \int_{r}^{2r} v\inv l(v)\,dv\le c_1 L(r).
 \]
If $R/2< r<R$, then $L(r)< L(R/2)\le c_3 L(R) < c_3 L(r)$, by (L$_3$).
\end{proof}  

\begin{lemma}\label{int-R2}
For all $x\in U_0$ and $0<r<R$, $\int_{\{|h|>1\wedge  R\}} K(x,h)\,dh\le c_3 L(r)$.
\end{lemma} 

\begin{proof} 
Let  $a:=1\vee R^{-2}$. 
By   (K) and  (L$_3$),   
\begin{equation*} 
\int_{\{|h|>1\wedge  R\}}     K(x,h) \,dh 
\le  a\int_{\reald} (1\wedge |h|^2)  K(x,h) \,dh \le a K_0\le c_3 L(R).
\end{equation*} 
for every $x\in U_0$. It remains to observe that $L(R)<L(r)$ for every  $r\in (0,R)$.
\end{proof} 

 \begin{proof}[Proof of Proposition \ref{6-km}.] 
(1) Let             $A$ be a Borel set in $V_r^c$.
If $y\in V_r$ and $z\in V_r^c$, then $|z-y|\le |z-x_0| +r\le 2|z-x_0|$,    hence
$ k( |z-x_0|)\le c_1 k(|z-y|/2)\le 2^dc_1^2 k(|z-y|)$,  by  (L$_1$).
So  (LS) implies that, for $t>0$,
\begin{eqnarray*} 
1&\ge &\mathbbm P^x[X_{\tau_r\wedge t}\in A]
  =  \mathbbm E^{x } \int_0^{ \tau_r\wedge t} \int_A   K(X_u, z-X_u)\,dz\,du\\
&\ge &   c_0\inv \mathbbm E^{x } \int_0^{ \tau_r\wedge t} \int_A   k(| z-X_u|)\,dz\,du
     \ge   (2^dc_0c_1^2)\inv \, \mathbbm E^x(\tau_r \wedge t)\int_A k(|z-x_0|) \, dz.
\end{eqnarray*} 
Since  $\int_{V_{R_0}\setminus V_r} k(|z-x_0|)\,dz=\kappa_d L(r)$, (\ref{tau-upper}) follows taking  $C_1:=2^d\kappa_d\inv c_0c_1^2$
and letting~$t$ tend to infinity.

Now let $r<s<R_0$ and  $a:=L(r)/ L(s) $. Then, 
\begin{equation}\label{muS}
 \kappa_d\inv \mu(S)=   \int _r^s  \frac{l(u)} {uL(u)}\,du=  - \ln L(u)|_r^s =\ln a,
\end{equation} 
and hence, if $A\subset S$,
\begin{equation*} 
 \int_A k(|z-x_0|)\,dz=\int_A L(|z-x_0|)\,d\mu(z)\ge L(s)\mu(A)\\
=  \kappa_d \, \frac{\ln a} a\, L(r)\,  \frac{\mu(A)}{\mu(S)}.
\end{equation*} 

(2) Let $0<2r\le s<R$. 
By (LS) (recall that $\tau_r<\infty$ $\mathbbm P^x$-a.s.\ by (1)), 
\[
        \mathbbm P^x[X_{\tau_r}\in V_s^c] =\mathbbm E^x\int_0^{\tau_r} \int_{V_s^c} K(X_u, z-X_u)\,dz\,du. 
\]
If $y\in V_r$, then $B(y, s/4) \subset V_{3s/4}$. Hence, by Lemmas \ref{L-doubling} and  \ref{int-R2},
\[
         \int_{V_s^c} K(y,z-y)\,dz\le   \int_{\{  s/4<|h|\}}  K(y,h)\,dh\le 
\kappa_d  c_0 L(s/4)+c_3\le   C_2 L(s),
\]
where $  C_2:=\kappa_d c_0C_4^2+c_3$. Thus $ \mathbbm P^x[X_{\tau_r }\in V_s^c] \le   C_2 L(s) \mathbbm E^x\tau_r$.

(3)
Let $\psi(u)\in \C^\infty(\real)$  
  such that $\psi(u)=u^2-2$
for every $u\in [-1,1]$, $\psi=0$ on~$\real\setminus [-2,2]$ and $-2\le \psi\le 0$.
 Let $0<r<  R$, $s:=1\wedge r$  and, for $y,h\in\reald$,
\[
f(y):=\psi(|y-x_0|/r) \und F(y,h):=(f(y+h)-f(y)) K(y,h).
\]
By (D), for every $x\in V_r$ and $t>0$,
\begin{equation}\label{ffr}
         \mathbbm  E^{x} f(X_{  \tau_r\wedge t}) -f(x)=
 \mathbbm E^{x}\int_0^{  \tau_r\wedge t}\int_{\reald} F(X_u,h) \,dh\,du. 
\end{equation} 
Let $y\in V_r$.
Since $  -2\le f\le 0$, we have $|F|\le 2 K$, and hence, by~(L$_0$) and Lemma~\ref{int-R2},
\begin{equation*} 
   \int_{\{s<|h|\}}  | F(y,h)|\,dh \le   2 \int_{\{r<|h|<R_0\}\cup \{ 1\wedge  R<|h|\}}   K(y,h)\,dh 
\le 2(\kappa_d c_0+c_3) L(r).
\end{equation*} 
 By (K$_0$), $K(y,h)=K(y,-h)$ for every $y\in B_1$. 
Hence, by   (L$_0$) and   (L$_2$), 
\begin{multline}\label{tilde-L}
\int_{\{|h|<s\}}  F(y,h) \,dh
=  \int_{\{|h|< s\}}  \bigl( f(y+h)-f(y)-\langle \nabla f(y),  h\rangle \bigr)   K(y,h)\,dh 
\\
    =  \int_{\{|h|< s\}}  \frac{|h|^2}{r^2}\, K(y,h)\, dh\le \kappa_d c_0
r^{-2} \int_0^r ul(u)\,du\le  \kappa_d c_0c_2 L(r).
\end{multline} 
Combining the preceding estimates we see that, defining $  C_3:=10\kappa_dc_0(1+c_2) +c_3$, 
\begin{equation}\label{int-reald}
\int_{\reald}    F(y,h)\,dh\le  (1/2) C_3L(r). 
\end{equation} 
    
Finally, let $x\in V_{r/2}$. Then, by  (\ref{ffr}) and (\ref{int-reald}), for every $t>0$, 
\begin{equation*} 
            \mathbbm  E^{x} f(X_{  \tau_r\wedge t})  -f(x) \le (1/2) C_3 L(r) \mathbbm E^{x}( \tau_r),
\end{equation*} 
where $ f(X_{\tau_r})-f(x)>1/2$ on $\{\tau_r<\infty\}$.  Letting $t\to \infty$ we hence see that 
\begin{equation*} 
       1/2<   \mathbbm E^{x} f(X_{ \tau_r}) -f(x)\le (1/2) C_3 L(r) \mathbbm E^{x}(\tau_r )
\end{equation*} 
completing the proof.
\end{proof} 

 \begin{remark}
{\rm
Suppose for a moment that instead of having (L$_2$), which is equivalent to $\limsup_{r\to 0} \tilde L(r)/L(r)<\infty$,
we would have $\lim_{r\to 0} \tilde L(r)/L(r)= \infty$,   the preceding proof could easily be modified 
(using the equalities in  (\ref  {tilde-L}) for $ r<1\wedge R$) to show that then $E^x\tau_r\approx \tilde L(r)\inv$ for   $ 0<r<R$
and $x\in V_{r/2}$. This is the case, if $l(u)=u^{-2}(\ln u\inv)^{-2}$ (see the end of Section \ref{examples}).
}
\end{remark}

To apply Theorem \ref{hoelder},  we define
\[
             \rho_0(x):=L(|x-x_0|)\inv, \qquad x\in X_0:=V_R.
\]
 Then, for every $0<r<R$,
\begin{equation}\label{VUr} 
 V_r=U_{L(r)\inv}.
\end{equation}

\begin{corollary}\label{J23} 
There exist $\a, a_0, \delta_0\in (0,1)$ and $C_0\ge 1$, which depend only on~$K_0,c_0,c_1,c_2,c_3$, satisfying {\rm (J$_1$)}  and {\rm (J$_2$)}.
\end{corollary} 

\begin{proof}
If $r,s\in (0,R)$, then
\begin{equation}\label{rs2}
            r<s/2 \quad \mbox{ provided } \quad L( r)> C_4 L(s), 
\end{equation} 
since $L$ is strictly decreasing and $  C_4 L(s)\ge L(s/2)$, by Lemma \ref{L-doubling}. 
We define
\begin{equation*} 
\a:=a_0:=C_4\inv,\qquad \delta_0:=(3 C_1C_3C_4)\inv, \qquad C_0:=C_1C_2.
\end{equation*} 
Now we fix  $0<t<L(R)\inv$. Given    $0<\g \le \a$, there are  $s,r \in (0,R)$ such that 
\begin{equation*} 
t=L( s)\inv  \und \g t=L(r)\inv .
\end{equation*} 
Then $  U_{\g t}=V_r\subset V_{s/2}$, by (\ref{rs2}).
By Proposition~\ref{6-km}, (1) and (2),  for every $x\in U_{\g t}$,
\[
             \mu_x^{U_{\g t}}(U_t^c)=\mu_x^{V_r}(V_s^c)=\mathbbm P^x[X_{\tau_r}\notin V_s] \le C_1C_2\,  {L(s)}/{L(r)}=C_1C_2 \g.
\]
So  (J$_2$) holds.

To prove (J$_1$) let  $\g=\a$. By (\ref{rs2}),  $U_{\a^2 t}\subset V_r$ (consider $\a t$ instead of $t$).
Finally, suppose that $A$ is a universally measurable set in $S:=U_t\setminus U_{\a t}=V_s\setminus V_r$
 and let~$\mu$ be as in Proposition \ref{6-km},1.
Then there exists a closed set~$F$ contained in~$A$ or in~$S\setminus A$ such that 
$\mu(F)> \mu(S)/3$.  
Since $L(s)=\a L(r)$,  Proposition \ref{6-km} shows  that, for  every $x\in  U_{\a^2t}\subset  V_{r/2}$,
\begin{equation*} 
      \mu_x^{U_{\a t}}(F)  = \mathbbm P^x[X_{\tau_r}\in F]\ge  ( C_1C_3)\inv  \frac {\ln (\a\inv )}{\a\inv} \frac{\mu(F)}{\mu(S)}    > \delta_0.
\end{equation*} 
 \end{proof} 

\begin{proof}[Proof of Theorem \ref{hoelder-levy}]  Let  $0<r<R$, $x\in V_r$,  and $h\in \H_b(B_r)$.
By  Corollary~\ref{J23}, Theorem \ref{hoelder} and (\ref{VUr}), 
\[
|h(x)-h(x_0)|\le C\|h\|_\infty (\rho_0(x)/L(r)\inv )^{\b} =
 C\|h\|_\infty L(r)^\b L(|x-x_0|)^{-\b}.
\]
\end{proof}

\begin{corollary}\label{km-corollary}  
Let $0<r<R$. Then, for all  $h\in \H_b(B_r)$ and  $x,y\in B_{r/3}$,
\begin{equation}\label{hoelder-kassmann-formula}
                   |h(x)-h(y)|\le c_1  C\|h\|_\infty L(r)^\b L(|x-y|)^{-\b}.
\end{equation} 
\end{corollary} 

\begin{proof} Let $x,y\in B_{r/3}$. Then  $x\in B(y,2r/3)\subset B_r$, and hence, by Theorem \ref{hoelder-levy}, 
\begin{equation*} 
                  |h(x)-h(y)|\le C   \|h\|_\infty L(2r/3)^\b L(|x-y|)^{-\b},
\end{equation*} 
where $L(2r/3)\le L(r/2)\le c_1 L(r)$.
\end{proof}

 \section{Examples}\label{examples}

  Our assumptions in Section \ref{application} are satisfied   under the main assumptions made in \cite{kassmann-mimica-final},
and hence Corollary \ref{km-corollary} implies the statement of \cite[Theorem 3]{kassmann-mimica-final} in the case $f=0$.

Indeed, our (K) and (L$_0$) are localized versions of~(A$_1$) and~(K$_0$) (which, incidentally, imply $(l_1)$).  
The second inequality in (L$_1$) amounts to~$(l_3)$ (with $R$ in place of~$R_0$). 
Property~$(l_2)$ means that $l(v)/l(u)\ge c_L (v/u)^{-\g} $ for all $0<u\le v<R_0$.
In~particular, it leads to $l(r/2)\le 2^\g c_L\inv l(r)$ for all $0<r<R_0$. 
Moreover, it also implies that, for some $c>0$, 
\begin{equation}\label{L-crit}
\int_0^r u l(u)\,du\le c  r^2 L(r) , \qquad 0<r<R,
\end{equation} 
 a~fact which  is   part of the proof of  \cite[(8)]{kassmann-mimica-final} without having been stated separately.
Indeed,  suppose that $(l_2)$ holds,
 let $a:= 1-(R/R_0)^\g$ (so that $a=1$ if~$R_0=\infty$) and  $0<r<R$.
 Then (see the proof of \cite[Lemma 7]{kassmann-mimica-final}),
\begin{eqnarray*} 
 \frac{L(r)}{l(r)}=\int_r^{R_0} u\inv\, \frac {l(u)}{l(r)}  \,du
 &\ge&  c_L r^\g\int_r^{R_0} u^{-1-\g}\,du\\
&= &\g\inv c_L\bigl( 1-(r/R_0)^\g) \ge \g\inv c_L a.
\end{eqnarray*} 
Further (see the proof of \cite[(8)]{kassmann-mimica-final}),
\begin{equation*} 
    \frac1{l(r)}   \int_0^r u l(u)\,du\le c_L\inv r^{\g} \int_0^r  u^{1-\g}\,du=(2-\g)\inv c_L\inv r^2.
\end{equation*} 
Thus (\ref{L-crit}) holds with $c:= a\inv c_L^{-2} \g/(2-\g)$. 

Finally, for  given $l$ and $R$, (L$_3$) is no problem. 

  Our assumptions are satisfied as well in the second part of  \cite{kassmann-mimica-final} 
(beginning with Section 5), where $l$ is assumed to be locally bounded and to vary regularly
at zero with index $-\a\in (-2,0]$. Implicitly, this has already been used in Section 6 of that paper 
and based on considerations in its Appendix. Thus Corollary \ref{km-corollary} also implies \cite[Theorem 12]{kassmann-mimica-final}.

We note, however, that in the case $l(u)=u^{-2} (\ln u\inv)^{-2}$ property (L$_2$) does not hold,
since an easy calculation shows that $L\approx l$, whereas $r^{-2} \int_0^r u l(u)\,du \approx r^{-2} (\ln r\inv)\inv
= l(r) \ln r\inv$.

\bibliographystyle{plain} 
\def\cprime{$'$} \def\cprime{$'$}

{\small \noindent 
Wolfhard Hansen,
Fakult\"at f\"ur Mathematik,
Universit\"at Bielefeld,
33501 Bielefeld, Germany, e-mail:
 hansen$@$math.uni-bielefeld.de}

\end{document}